\documentclass[12pt,twoside]{amsart}
\usepackage{amsmath, amsthm, amscd, amsfonts, amssymb, graphicx}
\usepackage{enumerate}
\usepackage[colorlinks=true,
linkcolor=blue,
urlcolor=cyan,
citecolor=red]{hyperref}
\usepackage{mathrsfs}
\newtheorem{theorem}{Theorem}[section]
\newtheorem{lemma}{Lemma}[section]
\newtheorem{remark}{Remark}[section]

\newtheorem{corollary}{Corollary}[section]

\numberwithin{equation}{section}

\begin{document}
\title{On the P\'olya-Szeg\"o operator inequality}
\author{Trung Hoa Dinh, Hamid Reza Moradi, and Mohammad Sababheh}
\maketitle
\begin{abstract}
In this paper, we  present generalized P\'olya-Szeg\"o type inequalities for positive invertible operators on a Hilbert space for arbitrary operator means between the arithmetic and the harmonic means. As applications, we present Operator Gr\"uss, Diaz--Metcalf and Klamkin--McLenaghan inequalities. 
\end{abstract}

\pagestyle{myheadings}
\markboth{\centerline {D.T. Hoa, H.R. Moradi \& M. Sababheh}}
{\centerline {On the P\'olya-Szeg\"o operator inequality}}
\section{ Introduction }
Let $\Phi $ be a positive linear map on $\mathcal{B}\left( \mathcal{H} \right)$; the algebra of all bounded linear operators on a Hilbert space $\mathcal{H}$. Ando \cite{8} proved the inequality
\begin{equation}\label{13}
\Phi \left( A\sharp B \right)\le \Phi \left( A \right)\sharp\Phi \left( B \right),
\end{equation}
for any positive linear map $\Phi $ and positive operators $A, B$, where ``$\sharp$'' is the geometric mean in the sense of Kubo-Ando theory \cite{10}.  That is,
\[A\sharp B={{A}^{\frac{1}{2}}}{{\left( {{A}^{-\frac{1}{2}}}B{{A}^{-\frac{1}{2}}} \right)}^{\frac{1}{2}}}{{A}^{\frac{1}{2}}}.\]
Speaking of means, the arithmetic mean   $A\nabla B$  and  the harmonic  mean   $A!B$ of two invertible positive operators $A,B\in\mathcal{B}(\mathcal{H})$ are  defined,  respectively,  by
\[A\nabla B=\frac{A+B}{2}\quad\text{ and }\quad A!B={{\left( \frac{{{A}^{-1}}+{{B}^{-1}}}{2} \right)}^{-1}}.\]
It is well known that $A!B\le A\sharp B\le A\nabla B$. In fact, if $\sigma$ is a symmetric operator mean (in the sense that $A\sigma B=B\sigma A$), then $A!B\leq A\sigma B\leq A\nabla B,$ for the invertible positive operators $A,B\in\mathcal{B}(\mathcal{H})$.

The operator P\'olya-Szeg\"o inequality presents a converse of Ando's inequality \eqref{13}, as follows.
\begin{theorem}
Let $\Phi $ be a positive linear map and $A,B\in \mathcal{B}\left( \mathcal{H} \right)$ be such that $mI\le A,B\le MI$ for some scalars $0<m<M$ ($I$ stands for the identity operator). Then
\begin{equation}\label{05}
\Phi \left( A \right)\sharp\Phi \left( B \right)\le \frac{M+m}{2\sqrt{Mm}}\Phi \left( A\sharp B \right).
\end{equation}
\end{theorem}
The inequality \eqref{05} was first proved in \cite[Theorem 4]{9} under the sandwich condition $sA\le B\le tA$ $\left( 0<s\le t \right)$ for matrices (see also \cite{bo}). 

Recall that a continuous real-valued function $f$ defined on an interval $J$ is said to be  operator monotone if $A\le B$ implies  $f\left( A \right)\le f\left( B \right)$ for all self-adjoint operators $A,B$ with spectra in $J$.  Very recently, Hoa et al. \cite[Theorem 2.12]{01} proved the following converse of \eqref{13} that extends \eqref{05}.
\begin{theorem}\label{19}
Let $\Phi $ be a positive linear map, $f$ be an operator monotone function on $[0,\infty)$, $\tau ,\sigma $ operator means such that $!\le \tau ,\sigma \le \nabla $, and $0<m<M$. Then for any positive operators $A,B$ satisfying $mI\le A,B\le MI$, the following inequality holds
\begin{equation}\label{06}
f\left( \Phi \left( A \right) \right)\tau f\left( \Phi \left( B \right) \right)\le \frac{{{\left( M+m \right)}^{2}}}{4Mm}f\left( \Phi \left( A\sigma B \right) \right).
\end{equation}
\end{theorem}

The first target of this article is to present a generalized form of P\'olya-Szeg\"o  inequality. In particular, we present relations between 
$$\Phi \left( f\left( A \right) \right)\tau \Phi \left( f\left( B \right) \right)\;\;{\text{and}}\;\; \Phi \left( f\left( A\sigma B \right) \right)$$ under the sandwich condition $sA\leq B\leq tA,$ for the operator monotone function $f$  and the symmetric operator means $\sigma,\tau$. Similar discussion will be presented for operator monotone decreasing functions. See Theorem \ref{16} below for the exact statements.

\section{Main results}
In this section we present relations between 
$$\Phi \left( f\left( A \right) \right)\tau \Phi \left( f\left( B \right) \right)\;\;{\text{and}}\;\; \Phi \left( f\left( A\sigma B \right) \right)$$
as generalized forms of P\'olya-Szeg\"o  inequality. Then we show some applications including Gr\"uss,  Diaz--Metcalf  and  Klamkin--McLenaghan type inequalities.

The first main result in this direction will be presented in Theorem \ref{16} below. However, we will need some lemmas first.

\begin{lemma}\label{e3}
Let $A,B\in \mathcal{B}\left( \mathcal{H} \right)$ such that $sA\le B\le tA$ for some scalars $0<s\le t$.
\begin{itemize}
	\item[(a)] If $st\ge 1$, then
	\begin{equation}\label{12}
	\frac{2}{\sqrt{s}+\sqrt{t}}A\nabla B\le A\sharp B\le \frac{\sqrt{s}+\sqrt{t}}{2}A!B.
	\end{equation}
	\item[(b)] If $st\le 1$, then
	\begin{equation}\label{0012}
	\frac{2\sqrt{st}}{\sqrt{s}+\sqrt{t}}A\nabla B\le A\sharp B\le \frac{\sqrt{s}+\sqrt{t}}{2\sqrt{st}}A!B.
	\end{equation}
\end{itemize}
\end{lemma}
\begin{proof}
By appealing to functional calculus, it suffices to show the corresponding scalar inequalities. We define $f\left( x \right):=\frac{x+1}{2\sqrt{x}}$ where $0<s\le x\le t$. It is straightforward to see that
$$f\left( x \right)\le \frac{1}{2}\max \left\{ \sqrt{s}+\frac{1}{\sqrt{s}},\sqrt{t}+\frac{1}{\sqrt{t}} \right\}.$$
Consequently, 
\begin{equation}\label{3}
\frac{x+1}{2}\le \left\{ \begin{array}{lr}
\frac{\sqrt{s}+\sqrt{t}}{2}\sqrt{x}&\text{ if }st\ge 1 \\ 
\frac{\sqrt{s}+\sqrt{t}}{2\sqrt{st}}\sqrt{x}&\text{ if }st\le 1 \\ 
\end{array} \right.
\end{equation}
for $0<s\le x\le t$, and
\begin{equation}\label{03}
\frac{{1}/_{x}+1}{2}\le \left\{ \begin{array}{lr}
\frac{\sqrt{s}+\sqrt{t}}{2}\frac{1}{\sqrt{x}}&\text{ if }st\ge 1 \\ 
\frac{\sqrt{s}+\sqrt{t}}{2\sqrt{st}}\frac{1}{\sqrt{x}}&\text{ if }st\le 1 \\ 
\end{array} \right.
\end{equation}
for $0<\frac{1}{t}\le \frac{1}{x}\le \frac{1}{s}$. Now, if $0<s\le x\le t$, the inequalities \eqref{3} and  \eqref{03} imply
\[\frac{2}{\sqrt{s}+\sqrt{t}}\left( \frac{x+1}{2} \right)\le \sqrt{x}\le \frac{\sqrt{s}+\sqrt{t}}{2}{{\left( \frac{{1}/{x}\;+1}{2} \right)}^{-1}}\]
whenever $st\ge 1$, and
\[\frac{2\sqrt{st}}{\sqrt{s}+\sqrt{t}}\left( \frac{x+1}{2} \right)\le \sqrt{x}\le \frac{\sqrt{s}+\sqrt{t}}{2\sqrt{st}}{{\left( \frac{{1}/{x}\;+1}{2} \right)}^{-1}}\]
whenever $st\le 1$. This completes the proof of the lemma. 
\end{proof}
\begin{remark}
The substitution  of $s=\frac{m}{M}$ and $t=\frac{M}{m}$ in Lemma \ref{e3} implies the celebrated result \cite[Theorem 13]{1}
\[\frac{2\sqrt{Mm}}{M+m}A\nabla B\le A\sharp B\le \frac{M+m}{2\sqrt{Mm}}A!B.\] 
\end{remark}

The next elementary lemma is given  for completeness.
\begin{lemma}\label{6}
	Let $\alpha \ge 1$.
	\begin{itemize}
		\item[(a)] If $f:[0,\infty)\to [0,\infty)$ is an operator monotone function, then
		\begin{equation*}
		f\left( \alpha t \right)\le \alpha f\left( t \right).	
		\end{equation*}
		\item[(b)] If $g:[0,\infty)\to [0,\infty)$ is an operator monotone decreasing function, then
		\begin{equation*}
		g\left( \alpha t \right)\ge \frac{1}{\alpha }g\left( t \right).
		\end{equation*}
	\end{itemize}
\end{lemma}
Now we are ready to prove the first main result,.

\begin{theorem}\label{16}
	Let $\Phi $ be a positive linear map, $\tau ,\sigma $ operator means such that $!\le \tau ,\sigma \le \nabla $, and let $A,B\in \mathcal{B}\left( \mathcal{H} \right)$ such that $s A \leq B \leq t A$ for some scalars $0 < s \leq t$.
	\begin{itemize}
		\item[(i)] If $f$ is an operator monotone increasing function on $[0,\infty)$, then
		\begin{equation}\label{2}
		\Phi \left( f\left( A \right) \right)\tau \Phi \left( f\left( B \right) \right)\le \left(\frac{\sqrt{s}+\sqrt{t}}{2}\right)^2 \Phi \left( f\left( A\sigma B \right) \right)
		\end{equation}
		whenever $st\ge 1$, and
		\[\Phi \left( f\left( A \right) \right)\tau \Phi \left( f\left( B \right) \right)\le {{\left( \frac{\sqrt{s}+\sqrt{t}}{2\sqrt{st}} \right)}^{2}}\Phi \left( f\left( A\sigma B \right) \right)\]
		whenever $st\le 1$.
		\item[(ii)] If $g$ is an operator monotone decreasing function on $[0,\infty)$, then
		\begin{equation}\label{4}
		\Phi \left( g\left( A\tau B \right) \right)\le \left(\frac{\sqrt{s}+\sqrt{t}}{2}\right)^2\Phi \left( g\left( A \right) \right)\sigma \Phi \left( g\left( B \right) \right)
		\end{equation}
		whenever $st\ge 1$, and
		\[\Phi \left( g\left( A\tau B \right) \right)\le {{\left( \frac{\sqrt{s}+\sqrt{t}}{2\sqrt{st}} \right)}^{2}}\left( \Phi \left( g\left( A \right) \right)\sigma \Phi \left( g\left( B \right) \right) \right)\]
		whenever $st\le 1$.
	\end{itemize}
\end{theorem}

\begin{proof}
First assume that $st\ge 1$. We observe that 
\begin{equation*}
\begin{aligned}
\Phi \left( f\left( A \right) \right)\tau \Phi \left( f\left( B \right) \right)&\le \Phi \left( f\left( A \right) \right)\nabla \Phi \left( f\left( B \right) \right) \quad \text{(since $\tau \le \nabla $)}\\ 
& =\Phi \left( f\left( A \right)\nabla f\left( B \right) \right) \\ 
& \le \Phi \left( f\left( A\nabla B \right) \right) \quad \text{(by \cite[Corollary 1.12]{2})}\\ 
& \le \Phi \left( f\left( \left( \frac{\sqrt{s}+\sqrt{t}}{2} \right)A\sharp B \right) \right) \quad \text{(by LHS of \eqref{e3})}\\ 
& \le \frac{\sqrt{s}+\sqrt{t}}{2}\Phi \left( f\left( A\sharp B \right) \right)\quad \text{(by Lemma \ref{6} (a))}.  
\end{aligned}
\end{equation*}
On the other hand,
\begin{equation}\label{18}
\begin{aligned}
 \Phi \left( f\left( A\sharp B \right) \right)&\le \Phi \left( f\left( \left( \frac{\sqrt{s}+\sqrt{t}}{2} \right)A!B \right) \right) \quad \text{(by RHS of \eqref{e3})}\\ 
& \le \frac{\sqrt{s}+\sqrt{t}}{2}\Phi \left( f\left( A!B \right) \right) \quad \text{(by Lemma \ref{6} (a))}\\ 
& \le \frac{\sqrt{s}+\sqrt{t}}{2}\Phi \left( f\left( A\sigma B \right) \right)\quad \text{(since $!\le \sigma $)}. 
\end{aligned}
\end{equation}
These two inequalities together imply \eqref{2}. This completes the proof of the case of operator monotone functions and $st\geq 1.$

Now assume that $g$ is operator monotone decreasing. We have
\begin{equation}\label{5}
\begin{aligned}
g\left( A \right)\sigma g\left( B \right)&\ge g\left( A\nabla B \right) \quad \text{(by \cite[Theorem 2.1]{3})}\\ 
& \ge g\left( \left( \frac{\sqrt{s}+\sqrt{t}}{2} \right)A\sharp B \right) \quad \text{(by LHS of \eqref{e3})}\\ 
& \ge \frac{2}{\sqrt{s}+\sqrt{t}}g\left( A\sharp B \right) \quad \text{(by Lemma \ref{6} (b))}. 
\end{aligned}
\end{equation}
On the other hand,
\begin{equation}\label{11}
\begin{aligned}
g\left( A\sharp B \right)&\ge g\left( \left( \frac{\sqrt{s}+\sqrt{t}}{2} \right)A!B \right) \quad \text{(by RHS of \eqref{e3})}\\ 
& \ge \frac{2}{\sqrt{s}+\sqrt{t}}g\left( A!B \right) \quad \text{(by Lemma \ref{6} (b)) }\\ 
& \ge \frac{2}{\sqrt{s}+\sqrt{t}}g\left( A\tau B \right)  \quad \text{(since $!\le \tau $)}.
\end{aligned}
\end{equation}
Combining \eqref{5} and \eqref{11} yields
\[g\left( A\tau B \right)\le \frac{{{\left( \sqrt{s}+\sqrt{t} \right)}^{2}}}{4}\left( g\left( A \right)\sigma g\left( B \right) \right).\]
Applying $\Phi $, we infer that
\[\begin{aligned}
 \Phi \left( g\left( A\tau B \right) \right)&\le \frac{{{\left( \sqrt{s}+\sqrt{t} \right)}^{2}}}{4}\Phi \left( g\left( A \right)\sigma g\left( B \right) \right) \\ 
& \le \frac{{{\left( \sqrt{s}+\sqrt{t} \right)}^{2}}}{4}\Phi \left( g\left( A \right) \right)\sigma \Phi \left( g\left( B \right) \right)\quad \text{(by \cite[Theorem 3]{8})}.  
\end{aligned}\]
This completes the proof for  operator monotone decreasing functions in case $st\geq 1.$

The proof of the case  $st\le 1$ is similar to that $st\ge 1$; except instead of inequality \eqref{12} we use the inequality \eqref{0012}.
\end{proof}

As an application of Theorem \ref{16}, we have the following Gr\"uss type inequalities.
\begin{corollary}\label{013}
		Let $\Phi $ be a positive linear map, $\tau ,\sigma $ operator means such that $!\le \tau ,\sigma \le \nabla $, and let $A,B\in \mathcal{B}\left( \mathcal{H} \right)$ be such that $mI\le A,B\le MI$ for some scalars $0<m<M$.
	\begin{itemize}
		\item[(i)] If $f$ is an operator monotone increasing function on $[0,\infty)$, then
\[\Phi \left( f\left( A \right) \right)\tau \Phi \left( f\left( B \right) \right)-\Phi \left( f\left( A\sigma B \right) \right)\le \frac{{{\left( M-m \right)}^{2}}}{4Mm}f\left( M \right).\]
		\item[(ii)] If $g$ is an operator monotone decreasing function, then
\[\Phi \left( g\left( A\tau B \right) \right)-\Phi \left( g\left( A \right) \right)\sigma \Phi \left( g\left( B \right) \right)\le \frac{{{\left( M-m \right)}^{2}}}{4Mm}g\left( m \right).\]
	\end{itemize}
\end{corollary}
\begin{proof}
	It follows from Theorem \ref{16} (i) that
	\begin{equation}\label{9}
\Phi \left( f\left( A \right) \right)\tau \Phi \left( f\left( B \right) \right)\le \frac{{{\left( M+m \right)}^{2}}}{4Mm}\Phi \left( f\left( A\sigma B \right) \right).	
	\end{equation}
Hence
\[\begin{aligned}
 \Phi \left( f\left( A \right) \right)\tau \Phi \left( f\left( B \right) \right)-\Phi \left( f\left( A\sigma B \right) \right)&\le \left( \frac{{{\left( M+m \right)}^{2}}}{4Mm}-1 \right)\Phi \left( f\left( A\sigma B \right) \right) \\ 
& \le \left( \frac{{{\left( M+m \right)}^{2}}}{4Mm}-1 \right)f\left( M \right)  
\end{aligned}\]
where in the first inequality we used \eqref{9} and the second inequality follows from the fact that $f\left( m \right)I\le f\left( A\sigma B \right)\le f\left( M \right)I$.

The other case can be obtained similarly by utilizing Theorem \ref{16} (ii).
\end{proof}
In \cite[Theorem 3]{5}, the inequality
\begin{equation}\label{mia_thm3}
\frac{\|g(A)\sharp g(B)\|}{\|A\sharp B\|}\leq 2S(h)^2\left\|\frac{g(A\sharp B)}{A\sharp B}\right\|
\end{equation}
was proved for the positive matrices $A,B$ satisfying $mI\leq A,B\leq MI$, the operator convex function $g:[0,\infty)\to [0,\infty)$ satisfying $g(0)=0$ and the Specht's ratio $S(h).$ Following the same ideas as in \cite{5} one can prove the following more general form, which then implies a refinement of \eqref{mia_thm3}.
\begin{corollary}\label{7}
Let $A,B\in \mathcal{B}\left( \mathcal{H} \right)$ be such that $s A \leq B \leq t A$ for some scalars $0 < s \leq t$ with $st \geq 1$, and let $g$ be an operator convex function with $g\left( 0 \right)=0$. Then for for any $\tau \ge \sharp$, $\sigma \le \sharp$ and for any unitarily invariant norm ${{\left\| \cdot \right\|}_{u}}$, 
\begin{equation}\label{10}
\frac{||g(A) \tau g(B))||_u}{|| A \tau B||_u} \le \left (\frac{\sqrt{s} + \sqrt{t}}{2}\right )^2 \left | \left| \frac{g(A \sharp B)}{A\sharp B} \right | \right |_u,
\end{equation}
and
\begin{equation}\label{011}
\frac{||g(A) \sharp g(B))||_u}{|| A \sharp B||_u} \le \left (\frac{\sqrt{s} + \sqrt{t}}{2}\right )^2 \left | \left| \frac{g(A \sigma B)}{A\sigma B} \right | \right |_u.
\end{equation}
In particular, if $!\le \tau ,\sigma \le \nabla $,
\begin{equation*}
\frac{||g(A) \tau g(B))||_u}{|| A \tau B||_u} \le \left (\frac{\sqrt{s} + \sqrt{t}}{2}\right )^4 \left | \left| \frac{g(A \sigma B)}{A\sigma B} \right | \right |_u.
\end{equation*}
\end{corollary}
\begin{proof}
By Theorem \ref{16},
$$
\frac{||g(A) \tau g(B))||_u}{|| A \tau B||_u} \le \left(\frac{\sqrt{s} + \sqrt{t}}{2}\right)^2 || g(A \sharp B)||_u.
$$ 
Consequently, the following double inequality is valid:
$$
\frac{||g(A) \tau g(B))||_u}{|| A \tau B||_u}  \le \left (\frac{\sqrt{s} + \sqrt{t}}{2}\right )^2 \frac{||g(A \sharp B))||_u}{|| A \sharp B||_u}  \le \left (\frac{\sqrt{s} + \sqrt{t}}{2}\right )^2 \left\| \dfrac{ g\left( A\sharp B \right) }{A\sharp B }\right\|_u.
$$

The second inequality is obtained by similar arguments.
\end{proof} 
The case $st\le 1$ in Corollary \ref{7} is also valid if we employ inequality \eqref{0012} instead \eqref{12}.
\begin{remark}
In the special cases when $s=\frac{m}{M}$, $t=\frac{M}{m}$, and $\tau =\sharp$ (resp. $\sigma =\sharp$), \eqref{10} (resp. \eqref{011}) reduces to
\begin{equation}\label{15}
\dfrac{\left\| g\left( A \right) \sharp g\left( B \right) \right\|_u}{\left\| A\sharp B \right\|_u} \leq 2\left(	\frac{M+m}{2\sqrt{Mm}}\right)^2 \left\| \dfrac{ g\left( A\sharp B \right) }{A\sharp B }\right\|_u.
\end{equation}
This shows that the inequality \eqref{15} is a refinement of \cite[Theorem 3]{5}. 	To see that \eqref{15} is a refinement of \cite[Theorem 3]{5}, one has to recall that $\frac{M+m}{2}\le S\left( \frac{M}{m} \right)\sqrt{Mm}$ (see \cite{4}).
\end{remark}

\begin{remark}
	By choosing $\Phi $ as an identity map, $s=\frac{m}{M}$, $t=\frac{M}{m}$, and $\tau =\sigma =\sharp$ in \eqref{18} and \eqref{5}, we have the following two cases:
	\begin{itemize}
	\item[(i)] If $f$ is an operator monotone increasing function, then
	\begin{equation}\label{1}
	f\left( A \right)\sharp f\left( B \right)\le \frac{M+m}{2\sqrt{Mm}}f\left( A\sharp B \right).
	\end{equation}
	\item[(ii)] If $g$ is an operator monotone decreasing function, then
	\begin{equation}\label{8}
	g\left( A\sharp B \right)\le \frac{M+m}{2\sqrt{Mm}}\left( g\left( A \right)\sharp g\left( B \right) \right).
	\end{equation}
\end{itemize}
 As  mentioned in \cite[Theorem 6]{iz}, if $A,B\in \mathcal{B}\left( \mathcal{H} \right)$ are two positive operators such that $A\le B$ and $mI\le A\le MI$ for some scalars $0<m<M$, then
\[{{A}^{2}}\le \frac{{{\left( M+m \right)}^{2}}}{4Mm}{{B}^{2}}.\] 
Now, by the substitutions $A\to f\left( A \right)\sharp f\left( B \right)$ and $B\to \frac{M+m}{2\sqrt{Mm}}f\left( A\sharp B \right)$ in the above discussion, we get
\[{{\left( f\left( A \right)\sharp f\left( B \right) \right)}^{2}}\le {{\left( \frac{{{\left( M+m \right)}^{2}}}{4Mm} \right)}^{2}}f{{\left( A\sharp B \right)}^{2}}.\]
A similar approach gives 
\[g{{\left( A\sharp B \right)}^{2}}\le {{\left( \frac{{{\left( M+m \right)}^{2}}}{4Mm} \right)}^{2}}{{\left( g\left( A \right)\sharp g\left( B \right) \right)}^{2}}.\]
\end{remark}

We conclude this article by showing operator Diaz--Metcalf and  Klamkin--McLenaghan inequalities.
\begin{theorem}
Let $\Phi $ be a positive linear map, $\tau ,\sigma $ operator means such that $!\le \tau ,\sigma \le \nabla $, and let $A,B\in \mathcal{B}\left( \mathcal{H} \right)$ such that $s A \leq B \leq t A$ for some scalars $0 < s \leq t$. If $f$ is a non-negative operator monotone function, then
\begin{itemize}
	\item (Operator   Diaz--Metcalf type inequality)
	\begin{equation}\label{21}
\Phi \left( f\left( \sqrt{st}A \right) \right)\tau \Phi \left( f\left( B \right) \right)\le {{\left( \frac{\sqrt{s}+\sqrt{t}}{2} \right)}^{2}}\Phi \left( f\left( A\sigma B \right) \right)
	\end{equation}
whenever $\sqrt{st}\ge 1$.	
	\[\Phi \left( f\left( \sqrt{st}A \right) \right)\tau \Phi \left( f\left( B \right) \right)\le \frac{{{\left( \sqrt{s}+\sqrt{t} \right)}^{2}}}{4\sqrt{st}}\Phi \left( f\left( A\sigma B \right) \right)\]
whenever $\sqrt{st}\le 1$.		
	\item (Operator   Klamkin--McLenaghan type inequality)
	\begin{equation}\label{22}
	\begin{aligned}
	& \Phi {{\left( f\left( A\sigma B \right) \right)}^{-\frac{1}{2}}}\Phi \left( f\left( B \right) \right)\Phi {{\left( f\left( A\sigma B \right) \right)}^{-\frac{1}{2}}}-\Phi {{\left( f\left( A\sigma B \right) \right)}^{\frac{1}{2}}}\Phi {{\left( f\left( \sqrt{st}A \right) \right)}^{-1}}\Phi {{\left( f\left( A\sigma B \right) \right)}^{\frac{1}{2}}} \\ 
	& \le \frac{{{\left( \sqrt{s}+\sqrt{t} \right)}^{2}}}{2}-2I-\left( {{\left( \Phi {{\left( f\left( A\sigma B \right) \right)}^{-\frac{1}{2}}}\Phi \left( f\left( \sqrt{st}A \right) \right)\Phi {{\left( f\left( A\sigma B \right) \right)}^{-\frac{1}{2}}} \right)}^{\frac{1}{2}}} \right. \\ 
	&\qquad {{\left. -{{\left( \Phi {{\left( f\left( A\sigma B \right) \right)}^{-\frac{1}{2}}}\Phi \left( f\left( \sqrt{st}A \right) \right)\Phi {{\left( f\left( A\sigma B \right) \right)}^{-\frac{1}{2}}} \right)}^{-\frac{1}{2}}} \right)}^{2}}
	\end{aligned}
	\end{equation}
whenever $\sqrt{st}\ge 1$.		
\begin{equation}\label{14}
\begin{aligned}
& \Phi {{\left( f\left( A\sigma B \right) \right)}^{-\frac{1}{2}}}\Phi \left( f\left( B \right) \right)\Phi {{\left( f\left( A\sigma B \right) \right)}^{-\frac{1}{2}}}-\Phi {{\left( f\left( A\sigma B \right) \right)}^{\frac{1}{2}}}\Phi {{\left( f\left( \sqrt{st}A \right) \right)}^{-1}}\Phi {{\left( f\left( A\sigma B \right) \right)}^{\frac{1}{2}}} \\ 
& \le \frac{{{\left( \sqrt{s}+\sqrt{t} \right)}^{2}}}{2\sqrt{st}}-2I-\left( {{\left( \Phi {{\left( f\left( A\sigma B \right) \right)}^{-\frac{1}{2}}}\Phi \left( f\left( \sqrt{st}A \right) \right)\Phi {{\left( f\left( A\sigma B \right) \right)}^{-\frac{1}{2}}} \right)}^{\frac{1}{2}}} \right. \\ 
&\qquad {{\left. -{{\left( \Phi {{\left( f\left( A\sigma B \right) \right)}^{-\frac{1}{2}}}\Phi \left( f\left( \sqrt{st}A \right) \right)\Phi {{\left( f\left( A\sigma B \right) \right)}^{-\frac{1}{2}}} \right)}^{-\frac{1}{2}}} \right)}^{2}} 
\end{aligned}
\end{equation}
whenever $\sqrt{st}\le 1$.
\end{itemize}
\end{theorem}
\begin{proof}
We assume $st\ge 1$. From the assumption $sA\leq B\leq tA$, it follows that $\sqrt{s}\le {{\left( {{A}^{-\frac{1}{2}}}B{{A}^{-\frac{1}{2}}} \right)}^{\frac{1}{2}}}\le \sqrt{t}$. Therefore, 
\begin{equation}\label{23}
\frac{\sqrt{st}A+B}{2}\le \left( \frac{\sqrt{s}+\sqrt{t}}{2} \right)A\sharp B.
\end{equation}
Now, since $f$ is an operator monotone increasing we have
\[\begin{aligned}
 \frac{f\left( \sqrt{st}A \right)+f\left( B \right)}{2}&\le f\left( \frac{\sqrt{st}A+B}{2} \right) \quad \text{(by \cite[Corollary 1.12]{2})}\\ 
& \le f\left( \left( \frac{\sqrt{s}+\sqrt{t}}{2} \right)A\sharp B \right) \quad \text{(by \eqref{23})}\\ 
& \le f\left( {{\left( \frac{\sqrt{s}+\sqrt{t}}{2} \right)}^{2}}A!B \right) \quad \text{(by RHS of \eqref{12})}\\ 
& \le {{\left( \frac{\sqrt{s}+\sqrt{t}}{2} \right)}^{2}}f\left( A!B \right) \quad \text{(by Lemma \ref{6}(a))}\\ 
& \le {{\left( \frac{\sqrt{s}+\sqrt{t}}{2} \right)}^{2}}f\left( A\sigma B \right) \quad \text{(since $!\le \sigma $)}.  
\end{aligned}\]
It follows from the linearity of $\Phi$ and the fact $\tau \le \nabla $ that
\[\begin{aligned}
 \Phi \left( f\left( \sqrt{st}A \right) \right)\tau \Phi \left( f\left( B \right) \right)&\le \frac{\Phi \left( f\left( \sqrt{st}A \right) \right)+\Phi \left( f\left( B \right) \right)}{2} \\ 
& \le {{\left( \frac{\sqrt{s}+\sqrt{t}}{2} \right)}^{2}}\Phi \left( f\left( A\sigma B \right) \right).  
\end{aligned}\]
So we have \eqref{21}. The case $st\le 1$ can be obtained similarly.

From \eqref{21} we easily infer that
\begin{equation}\label{24}
\Phi \left( f\left( \sqrt{st}A \right) \right)+\Phi \left( f\left( B \right) \right)\le \frac{{{\left( \sqrt{s}+\sqrt{t} \right)}^{2}}}{2}\Phi \left( f\left( A\sigma B \right) \right).
\end{equation}
The estimate \eqref{24} guarantees
\[\begin{aligned}
& \Phi {{\left( f\left( A\sigma B \right) \right)}^{-\frac{1}{2}}}\Phi \left( f\left( B \right) \right)\Phi {{\left( f\left( A\sigma B \right) \right)}^{-\frac{1}{2}}} \\ 
& \le \frac{{{\left( \sqrt{s}+\sqrt{t} \right)}^{2}}}{2}-\Phi {{\left( f\left( A\sigma B \right) \right)}^{-\frac{1}{2}}}\Phi \left( f\left( \sqrt{st}A \right) \right)\Phi {{\left( f\left( A\sigma B \right) \right)}^{-\frac{1}{2}}}.
\end{aligned}\]
We set
\[\begin{aligned}
& X:=\Phi {{\left( f\left( A\sigma B \right) \right)}^{-\frac{1}{2}}}\Phi \left( f\left( B \right) \right)\Phi {{\left( f\left( A\sigma B \right) \right)}^{-\frac{1}{2}}} \\ 
&\qquad -\Phi {{\left( f\left( A\sigma B \right) \right)}^{\frac{1}{2}}}\Phi {{\left( f\left( \sqrt{st}A \right) \right)}^{-1}}\Phi {{\left( f\left( A\sigma B \right) \right)}^{\frac{1}{2}}}
\end{aligned}\]
and observe
\begin{equation}\label{17}
X\le \frac{{{\left( \sqrt{s}+\sqrt{t} \right)}^{2}}}{2}-T-{{T}^{-1}}
\end{equation}
where $T=\Phi {{\left( f\left( A\sigma B \right) \right)}^{-\frac{1}{2}}}\Phi \left( f\left( \sqrt{st}A \right) \right)\Phi {{\left( f\left( A\sigma B \right) \right)}^{-\frac{1}{2}}}$. Notice that
\begin{equation}\label{20}
T+{{T}^{-1}}={{\left( {{T}^{\frac{1}{2}}}-{{T}^{-\frac{1}{2}}} \right)}^{2}}+2I.
\end{equation}
Combining \eqref{17}and \eqref{20} we get
\[X\le \frac{{{\left( \sqrt{s}+\sqrt{t} \right)}^{2}}}{2}-{{\left( {{T}^{\frac{1}{2}}}-{{T}^{-\frac{1}{2}}} \right)}^{2}}-2I.\] 
which is equivalent with the  inequality \eqref{22}. The inequality \eqref{14} is obtained by similar arguments.
\end{proof}
\begin{remark}
Assume $\sqrt{st}\ge 1$. Due to the monotonicity property of  operator means, we have
\[\Phi \left( f\left( A \right) \right)\tau \Phi \left( f\left( B \right) \right)\le \Phi \left( f\left( \sqrt{st}A \right) \right)\tau \Phi \left( f\left( B \right) \right)\le {{\left( \frac{\sqrt{s}+\sqrt{t}}{2} \right)}^{2}}\Phi \left( f\left( A\sigma B \right) \right)\]
which is stronger than \eqref{2}.
\end{remark}

\vskip 0.4 true cm

\tiny(Trung Hoa Dinh)  Division of Computational Mathematics and Engineering, Institute for Computational Science, Ton Duc Thang University, Ho Chi Minh City, Vietnam; \\ Faculty of Civil Engineering, Ton Duc Thang University, Ho Chi Minh City, Vietnam. \\
Department of Mathematics and Statistics, University of North Florida, USA.

{\it E-mail address:} dinhtrunghoa@tdt.edu.vn

\vskip 0.4 true cm

\tiny(Hamid Reza Moradi) Young Researchers and Elite Club, Mashhad Branch, Islamic Azad University, Mashhad, Iran.

{\it E-mail address:} hrmoradi@mshdiau.ac.ir

\vskip 0.4 true cm

\tiny(Mohammad Sababheh) Department of Basic Sciences, Princess Sumaya University for Technology, Amman 11941, Jordan.

{\it E-mail address:} sababheh@yahoo.com; sababheh@psut.edu.jo
\end{document}